\documentclass[10pt]{amsart}
\usepackage{amsfonts,amsmath,mathrsfs,amssymb,amsthm,amscd,latexsym}

\usepackage[usenames]{color}
\usepackage[all]{xy}
\usepackage{enumerate}
\xyoption{all}
\usepackage{srcltx} 

\usepackage{pgf,tikz}
\usetikzlibrary{arrows}
\newtheorem{thm}{Theorem}[section]
\newtheorem{lem}[thm]{Lemma}
\newtheorem{cor}[thm]{Corollary}
\newtheorem{prop}[thm]{Proposition}

\theoremstyle{definition}
\newtheorem{rem}[thm]{Remark}

\newcommand{\cD}{{\mathcal D}}
\newcommand{\cM}{{\mathcal M}}

\newcommand{\cJ}{{\mathcal J}}

\newcommand{\cO}{{\mathcal O}}

\newcommand{\cZ}{{\mathcal Z}}
\newcommand{\cC}{{\mathcal C}}

\newcommand{\bC}{{\mathbb C}}
\newcommand{\bZ}{{\mathbb Z}}
\newcommand{\tC}{{\widetilde C}}

\newcommand{\lra}{\longrightarrow}
\newcommand{\m}[1]{\mathcal{#1}}

\newcommand{\Q}{{\mathbb Q}}
\newcommand{\pa}[1]{{\left(#1\right)}}

\title{Isogenies of Jacobians}
 \author{V. Marcucci}
 \address{Valeria Marcucci \\  Universit\`a degli Studi di Pavia  \\ Dipartimento di Matematica \\ Via Ferrata, 1  \\ 27100 Pavia, Italy  }
 \email{valeria.marcucci@unipv.it}

 \author{J.C. Naranjo}
 \address{Juan Carlos Naranjo  \\Universitat de Barcelona  \\ Facultat de Matem\`atiques \\ Gran Via 585  \\ 08007 Barcelona, Spain  }
 \email{jcnaranjo@ub.edu}

 \author{G.P. Pirola}
 \address{Gian Pietro Pirola  \\ Universit\`a degli Studi di Pavia  \\ Dipartimento di Matematica \\ Via Ferrata, 1  \\ 27100 Pavia, Italy  }
 \email{gianpietro.pirola@unipv.it}

 \thanks{Naranjo has been partially supported by the Proyecto de Investigaci\'on MTM2012-38122-C03-02; Pirola
has been partially supported by Gnsaga and by MIUR PRIN 2012: \textup{Moduli, strutture geometriche e loro applicazioni}. Part of this research has been done during the visit of Pirola to the IMUB in the spring of 2013.}

\begin{document}

\begin{abstract}
We prove by means of the study of the infinitesimal variation of Hodge structure and a generalization of the classical Babbage-Enriques-Petri theorem that the Jacobian variety of a generic element of a codimension $k$
subvariety of $\mathcal M_g$  is not isogenous to a distinct Jacobian if $g>3k+4$. We extend this result to $k=1, g\ge 5$ by using degeneration methods. 
\end{abstract}

\maketitle

\pagestyle{plain}

\section{Introduction}

Let $\cZ$ be a subvariety of the moduli space $\cM_g$  of complex smooth curves of genus $g$ of codimension $k>0$. We want to show that under some numerical restrictions, 
the Jacobian of a generic element of $\cZ$ is not isogenous to a distinct Jacobian. In other words, all the curves of genus $g$ contained in $JC$, with $C$
generic in $\cZ$, are birationally equivalent. 
This is an extension of the Theorem proved by Bardelli and Pirola (see \cite{BarPir}) for the whole $\mathcal M_g$ and can be seen as a Noether-Lefschetz locus problem for 
surfaces which are the product of two curves of the same genus (see Corollary (\ref{remark_NL})). More precisely, our result is as follows:
\vskip 3mm
\begin{thm}\label{main}
 Let $\cZ \subset \mathcal M_g$ a codimension $k>0$ subvariety. 
Assume  that $g>3k+4$ (in particular $g>7$), then the Jacobian of a generic curve $C$ of $\cZ$ is not isogenous to another Jacobian.
 The same is true for $k=1$ and $g\ge 5$.
\end{thm}
Observe that the Theorem fails for $g=4$ and $k=1$: in this case $\m{M}_4$ is a divisor in $\m{A}_4$, therefore intersecting in the Siegel upper space $\m{H}_4$ the 
Jacobian locus $\m{J}_4$ with the image $j(\m{J}_4)$ by the action of a fixed isogeny $j$, we get a divisor in $\m{M}_4$ where the Jacobian of a generic element is 
isogenous to a different Jacobian. 

For $g>3k+4$, our strategy is as follows: after a base change we have two families of smooth complex curves of genus $g$ 
on a base variety $W,$ $\pi: \mathcal C\to W$ and $\pi' :\mathcal C' \to W,$ and a family of isogenies of the associated family of Jacobians, that is
\[
\chi : J(\mathcal C')\to J(\mathcal C).
\]
This means that for $t\in W$ the map $\chi_t: J(C'_t)\to J(C_t)$ is an isogeny, here $C'_t=\pi'^{-1}(t)$ and
$C_t=\pi^{-1}(t).$

It follows that the associated rational Hodge structures are isomorphic. Consider the local (polarized) systems
$\Lambda_\bZ=R^1\pi_\ast \bZ$ and $\Lambda'_{\bZ}=R^1\pi'_\ast \bZ$  and tensoring by $\bC$ $\Lambda_{\bC}=R^1\pi_\ast \bC$ and $\Lambda'_{\bC}=R^1\pi'_\ast \bC.$

In particular the infinitesimal variation of Hodge structure associated to the Hodge filtration of
$ \Lambda^{1,0}\subset \Lambda_{\bC}$ and $\Lambda'^{1,0}\subset \Lambda'_{\bC}$ are isomorphic. We borrowed this basic observation  from Claire Voisin (see Remark (4.2.5) 
in \cite{BarPir}).  It is well known 
 that the infinitesimal invariant of Hodge structure of curves determines the quadrics that contain a canonical curve (see \cite{CGGH}). This allows to translate our problem to a geometric 
 one. Let $I(2)$ and $J(2)$ be the space of quadrics that contain the canonical curve associated to $C_t$ and $D_t.$
It follows that under a choice of a suitable canonical embedding $I(2)\cap J(2)$ has codimension $\geq k$, where $k$ is the codimension of $m(W)$ in 
$\cM_g$ and $m$ is the modular mapping $m:W\to \cM_g.$
We can bound the codimension $k$ by using the Clifford index. For this we prove a result that gives an interesting (at least in our opinion) reconstruction result of the 
curve from a {\em partial system of quadrics}. It is a generalization of the Babbage-Enriques-Petri theorem (see e.g. chapter $3$, section $3$ in \cite{ACGH}). 

\vskip 3mm
\begin{thm} \label{main_geom}
Let $C$ be a curve of genus $g$ and Clifford index $c$. Let $I_2\subset Sym^2 H^0(C,\omega_C)$ be the vector space of the equations of the 
quadrics containing $C$ and let $K\subset I_2$ be a linear subspace of codimension $k$. If $g>2k+5$ and $c>k+1$, then $C$ is the only 
irreducible non-degenerate curve contained in the intersection of the quadrics of $K$.   
\end{thm}

\begin{cor} \label{main_cor}
 Let $C$ be a generic curve in a codimension $k$ subvariety $\cZ$ of $\mathcal M_g$. Let $I_2\subset Sym^2 H^0(C,\omega_C)$ be the vector space of the equations of the 
 quadrics containing $C$. Let $K\subset I_2$ be a linear subspace of codimension $k$. If $g>3k+4$, then $C$ is the only irreducible non-degenerate curve contained in the 
 intersection of the quadrics of $K$.
\end{cor}

The Corollary is a consequence of the Theorem (\ref{main_geom}). Indeed, let $c$ be the Clifford index of a generic element of $\cZ$. The locus 
of curves with given Clifford index can have several components and the minimal codimension is attained when $c$ is realized by a $g^1_d$ linear series, with $c=d-2$. This  follows easily by a parameter count from a result in  \cite{CopMar} where it is proved that a curve $C$ with Clifford index $c$ is either $(c+2)$-gonal or the dimension of the Brill-Noether locus $W^1_{c+3}(C)$ is at least $1$. 

Then, by Riemann-Hurwitz, the codimension of the
component of the curves with a $g^1_{c+2}$ linear series is $3g-3-(2g-2+2(c+2)-3)=g-2c-2$. Hence, since we assume $g>3k+4$, we have
\begin{equation*}
 k\ge g-2c-2 >3k+4-2c-2=3k-2c+2.
\end{equation*}
  Therefore $   c>k+1$
 and the result follows from (\ref{main_geom}) since $g>3k+4$ implies $g>2k+5$ for $k>0$.

In section $2$ we start the proof of  Theorem (\ref{main}) under the hypothesis $g>3k+4$ by reducing it to Corollary (\ref{main_cor}) following Voisin's observation indicated above. Theorem 
(\ref{main_geom}) will be proved in Section 3. The idea of the proof is as follows: assuming the existence of a second non-degenerate curve in the intersection of the quadrics we select linearly independent points $x_i$ in this curve. Then, by constructing a suitable rank $2$ vector bundle on $C$, we are able to find points $p_j\in C$ such that the linear span of the points $x_i$ is contained in the linear span of the points $p_j$. From this it is easy to obtain a contradiction by using a theorem of Ran \cite{Ran}.

\vskip 5mm
To prove the divisorial case of the main theorem, we use the original approach in \cite{BarPir}, based on the analysis of the map
\[
\chi_\bZ: \Lambda_\bZ\to \Lambda'_\bZ
\]
(in fact we will work with the dual lattices, that is, with the homology groups).
If we can prove that $\chi_\bZ(\Lambda_\bZ)=n\Lambda'_\bZ,$ $n\in \bZ$ we will get that $\cC'_t$ is isomorphic to
$\cC_t$ and $\chi_t$ is given by multiplication by $n.$
We use degeneration to $\Delta_0,$  and study the monodromy action on $\Lambda_\bZ.$
The basic geometric information is now encoded on the  generalized Jacobians. Roughly speaking one has to prove that part of the limit map $\chi_0$ is  multiplication by $n.$ This gives that the map $\chi_t$ 
is  multiplication by $n$ on
 {\em a part of }{\em the invariant cycles.}  We need finally  to have degenerations with independent monodromy to complete the proof.
It is clear that to follow this strategy one needs  
to control the  degeneration type. Using the theory of divisors on $\cM_g$ and following a valuable suggestion of Gavril Farkas,
  we realize  the above program  when $c=1.$ 
The degeneration procedure is performed in sections $4$, $5$ and $ 6.$ We will prove the existence of a type of degeneration  (if $c=1$)
to the union of a curve of genus $g-2$  and two  generic elliptic tails.
The {\em independent} degenerations to  $\Delta_0$ are obtained by letting the elliptic tails becomes singular. To extract more information from the degeneration we analyze the type of monodromy involved that we classify in three cases (a,b,c of section $5$). Then we analyze the geometry of the generalized
Jacobians by comparing their extension classes.
  In section $6$ we complete the proof by comparing the invariants of the two degenerations.

\vskip 3mm
\textbf{Acknowledgements:} We thank Gavril Farkas for useful and stimulating conversations on the moduli space of curves.  

\vskip 3mm
\section{Reduction to a problem on quadrics trough the canonical curve}

The aim of this section is to prove that Corollary (\ref{main_cor}) implies Theorem (\ref{main}) under the hypothesis $g>3k+4$ and $k>0$.

Remember that  an isogeny $\chi : A'\lra A$ between principally polarized Abelian varieties $(A',L_{A'})$ and $(A,L_{A})$ such that $\chi ^* L_A \cong L_{A'}^{\otimes m}$ 
is determined by a subgroup $H$ of the group of $m$-torsion points  $A'_m$ totally isotropic with respect to the Riemann bilinear form 
\[
 e_m:A'_m\times A'_m \lra \mu_m
\]
(being $A=A'/H$) and a level subgroup $\tilde H $ of the theta group  
$\mathcal G(L_{A'}^{\otimes m})$, see (\cite{MAV}, chapter 23). Then the moduli space of those isogenies can be rewritten as
\[
 \widetilde {\mathcal  A_g}^m=\{\chi:A'\lra A,\,\chi ^*(L_A)\cong L_{A'}^{\otimes m}\}_{/\cong} = \{(A',L_{A'};H,\tilde H)\}_{/\cong} 
\]
and the forgetful map is a finite covering $\varphi :\widetilde {\mathcal A_g}^m \lra \mathcal A_g$. Moreover the map $\psi: \widetilde {\mathcal A_g}^m \lra \mathcal A_g$ 
sending 
$ \chi:A'\lra A$ to $(A,L_A)$ is another covering space.
\vskip 3mm
Given a generic isogeny $\chi:A'\lra A$ we consider tangent spaces in the following diagram:
\[
 \xymatrix@C=1.pc@R=1.8pc{
  \widetilde {\mathcal  A_{g}}^m \ar[rr]^{\psi} \ar[d]_{\varphi }& &\mathcal A_g \\
 \mathcal A_g&
}
\]
and we get an isomorphism $\lambda $ as follows:
\[
 \xymatrix@C=1.pc@R=1.8pc{
 T_{\widetilde {\mathcal  A_{g}}^m,\chi} \ar[rr]^{d\psi}_{\cong} \ar[d]_{d\varphi }^{\cong }& &T_{ \mathcal A_g,A} \ar@{=}[r] & Sym^2H^0(A,T_A) \\
 T_{\mathcal A_g,A'}  \ar@{=}[d] \ar[rru]_{\lambda} && \\
Sym^2H^0(A',T_{A'}) &&&
}
\]

\vskip 3mm
Coming  back to our problem let us assume that the locus of curves in $\mathcal M_g$ with Jacobian isogenous to the Jacobian of some curve in $\cZ$ contains a codimension 
$k$ component $\cZ'\subset \mathcal M_g$. Our hypothesis on $k$ implies that a generic element $C'\in \cZ'$ satisfies $End(JC)\cong  \mathbb Z$ (see \cite{CGT} or \cite{baseN}).   Therefore an isogeny 
$\chi: JC' \lra JC$ must satisfy that the pull-back of the principal polarization in $JC$ is a multiple of the principal polarization in $JC'$. Hence there exists an 
integer $m$  
and an irreducible variety $\mathcal R\subset \widetilde {\mathcal A_g}^m$ dominating $\cZ'$ and $\cZ$ through $\varphi $ and $\psi$ respectively. Set $\m{M}:=\varphi 
^{-1}(\m{M}_g)$ and $\m{M}':=\psi ^{-1}(\m{M}_g)$. Then $\mathcal R \subset \m{M} \cap \m{M}'$. Fix a generic element $\chi:JC' \lra JC$ in $\mathcal R$. In the following 
diagram we consider in the first row the natural inclusions of tangent spaces at $\chi$ and we put in the second row its image by $d\varphi$:
\[
\xymatrix@C=1.pc@R=1.8pc{
 T_{\mathcal R, \chi}   \ar @{^{(}->}[r] \ar[d]^{\cong} & T_{\mathcal {M},\chi}  \ar @{^{(}->}[r] \ar[d]^{\cong} & T_{\mathcal {M},\chi}+ T_{\mathcal {M}',\chi} \ar 
 @{^{(}->}[r] \ar@{=}[d]
& T_{\tilde {\mathcal {A}}_g^m,\chi} \ar[d]^{\cong} \\
 T_{\mathcal Z, JC}   \ar @{^{(}->}[r]   & T_{\mathcal {M}_g,JC}=H^0(C,\omega_C^{\otimes 2})^* \ar @{^{(}->}[r]&\bar T \ar @{^{(}->}[r]& Sym ^2 H^0(C, \omega _C) ^* 
}
\]

Observe that, by the Grassmann formula, the dimension of $\bar T$ is at most $3g-3+k$.
Set $K(C):=\text{Kernel}(Sym^2 H^0(C,\omega _C) \lra \bar T^*)$, this is a subspace of the vector space $I_2(C)$ of the quadrics containing the image of $C$ by the 
canonical  map. The codimension of $K(C)$ in $I_2(C)$ is at most $k$. By using $\psi$ instead of $\varphi $ we get the corresponding vector space $K(C')\subset I_2(C')$ 
and we obtain a canonical isomorphism $K(C)\cong K(C')$. Then the Corollary (\ref{main_cor}) implies that $C$ and $C'$ are isomorphic and, since $End(JC)=\mathbb Z$, the 
isogeny is a multiple of the identity. 

\vskip 3mm
\section{High codimension family of quadrics through the canonical curve}

This section is devoted to the proof of Theorem (\ref{main_geom}). We fix the notation $K\subset I_2$ of the statement. We assume that the intersection of all the 
quadrics  of $K$ contains  an irreducible non-degenerate curve different from $C$. In particular we can select $k+1$ linearly independent points $x_i \in \bigcap _{Q\in 
K}Q \subset \mathbb PH^0(C,\omega_C)^*$ such that $x_i \notin C$.  We choose a representative of $x_i$ in $H^0(C,\omega_C)^*$ and we denote it with the same  symbol. Then 
$
x_i\otimes x_i \in Sym ^2 H^0(C,\omega _C)^*. 
$
 We denote by $L$ the linear variety spanned by these points.
\vskip 3mm
Let $R, R'$ be the quotients $I_2/K$  and  $Sym^2 H^0(C,\omega_C)/K$ respectively. Then we have the diagram of vector spaces
\begin{equation*}
\xymatrix@C=1.pc@R=1.8pc{
& & 0\ar[d] &0\ar[d] &\\
& 0\ar[r] & K \ar[d] \ar[r]  & I_2 \ar[d] \ar[r] & R \ar[r]   &  0 \\
&  & Sym^2H^0(C,\omega_C)\ar@{=}[r]\ar[d]  & Sym^2H^0(C,\omega_C)\ar[d] & &\\
0\ar[r] &R\ar[r]  &R' \ar[r] \ar[d] & H^0(C,\omega_C^{\otimes 2}) \ar[r] \ar[d] &0 & \\
 && 0&0&&
}\end{equation*}
and its dual 
\begin{equation*}
\xymatrix@C=1.pc@R=1.8pc{
 & & 0 &0\ &\\
 0\ar[r] & R^*  \ar[r]  & I_2^* \ar[u] \ar[r] & K^* \ar[r] \ar[u]   &  0 \\
 & & Sym^2H^0(C,\omega_C)^*\ar@{=}[r]\ar[u]  & Sym^2H^0(C,\omega_C)^*\ar[u] & &\\
&0\ar[r] &H^1(C,T_C)= H^0(C,\omega_C^{\otimes 2})^*\ar[u] \ar[r]  &R'^* \ar[r] \ar[u] & R^* \ar[r]  &0 & \\
 && 0\ar[u]&0\ar[u]&&
}\end{equation*}
Since all the quadrics of $K$ vanish on $x_i$ then the image of $L$ in $K^*$ is zero, hence $L\subset R'^*$. Since $L$ has dimension $k+1$ and, by the hypothesis on $K$, 
$\dim R=k$ we have $H^1(C,T_C)\cap L \neq (0).$  Let $\alpha$ be a non-trivial element in this intersection. Looking at $H^1(C,T_C)=Ext^1(\omega_C,\cO_C)$ as classes 
of extensions we associate to $\alpha$ a rank $2$ vector bundle $E_{\alpha }$ and a short exact sequence:
\begin{equation*}
 0\lra \cO_C \lra E_{\alpha }  \lra \omega_C \lra 0.
\end{equation*}
 The coboundary map $H^0(C,\omega_C) \lra H^1(C,\cO_C)$ is the cup-product with $\alpha $. Since $\alpha \in L$, then $\alpha =\sum_{i=1}^{k+1}a _i \, x_i \otimes x_i$. 
 Therefore, denoting by $H_i$ the kernel of the form $x_i:H^0(C,\omega_C)\lra \mathbb C$, the intersection $H_1\cap \dots \cap H_{k+1}$ is contained in 
$Ker(\cdot \cup \alpha)$, in fact
\begin{equation}\label{Kernel_cup}
  Ker(\cdot \cup \alpha) = \bigcap _{i \text{ with }a_i\ne 0}H_i.
\end{equation}
We can assume that $x_1,\dots,x_{k'}$, $k'\le k+1$, are the points such that $a_i\ne0$. Then there are  $g-k'$ sections of $H^0(C,\omega_C)$ lifting to $E_{\alpha }$. 
Let $W\subset H^0(C,E_{\alpha})$ be the vector space generated by these sections.
We consider the wedge product of sections:
\begin{equation*}
\psi :\Lambda ^2W \hookrightarrow \Lambda ^2 H^0(C, E_{\alpha}) \lra H^0(C, \det  E_{\alpha }) \,=\, H^0(C,\omega_C).
\end{equation*}
The hypothesis $g>2k+5$ implies that the projectivization of the kernel of $\psi $ (which has codimension at most $g$) intersects in $\mathbb P (\Lambda ^2 W)$ the Grassmannian of the decomposable elements. Hence there are two sections $s_1,s_2\in W\subset H^0(C,E_{\alpha})$ such that $s_1\wedge s_2=0$.
This means that they generate 
a rank $1$ torsion free sheaf $M_{\alpha }\subset E_{\alpha}$, hence a line bundle. By construction $h^0(C,M_{\alpha})\ge 2$. Let us consider the quotient of sheaves $Q_{\alpha}=E_{\alpha}/M_{\alpha}$ and let $Q^0_{\alpha} $ be the quotient of $Q_{\alpha}$ by its torsion subsheaf. The kernel $L_{\alpha}$ of the natural map $E_{\alpha}\mapsto Q^0_{\alpha}$ is a line bundle that by construction contains $M_{\alpha}$. In particular  $h^0(C,L_{\alpha})\ge 2$.
Observe that $Q^0_{\alpha}\cong \omega _C \otimes L_{\alpha }^{-1}$. We get a diagram:
\begin{equation}\label{diagram}
\xymatrix@C=1.pc@R=1.8pc{
 & & 0\ar[d] & &\\
& & L_{\alpha } \ar[d] & & \\
0 \ar[r] & \cO_C \ar[rd]^{\rho} \ar[r] &E_{\alpha } \ar[r] \ar[d]  &\omega_C \ar[r] &0 \\
 & & \omega _C \otimes L_{\alpha }^{-1} \ar[d] & & \\
 && 0 &&
}\end{equation}
Note that $\rho \ne 0$, otherwise the section of $E_{\alpha}$ represented by the horizontal arrow $\cO \mapsto E_{\alpha}$ would belong to $W$ which contradicts the definition of $W$.

\vskip 3mm
Observe that the existence of the map $\rho $ implies that $h^0(C,\omega_C\otimes L_{\alpha }^{-1})$ is positive
We distinguish two cases:

\vskip 3mm
Case 1: $h^0(C,\omega_C\otimes L_{\alpha }^{-1}) \ge 2$. Then we can use $L_{\alpha}$ to compute the Clifford index of the curve. We have:
\begin{equation*}
 h^0(C,L_{\alpha})+h^0(C,\omega_C\otimes L_{\alpha}^{-1})\ge h^0(C,E_{\alpha }) \ge g-k'+1
\end{equation*}
that combined with Riemann-Roch gives $2 h^0(C,L_{\alpha})\ge deg(L_{\alpha})+2-k'$. Therefore 
\begin{equation*}
  deg(L_{\alpha}) -2 h^0(C,L_{\alpha})+2 \le k'\le k+1
\end{equation*}
which is contradiction, since $k+1<c$ by hypothesis.

\vskip 3mm
Case 2:  $h^0(C,\omega_C\otimes L_{\alpha }^{-1}) = 1$. Then $h^0(C,L_{\alpha})\ge g-k'$. Let $e$ be the degree of $\omega_C\otimes L_{\alpha}^{-1}\cong \cO_C(p_1+\dots 
+p_e)$. 
\vskip 3mm 
Claim: we have that $e\le k'$. 
\vskip 5mm
Indeed, since $\rho$ induces an isomorphism $H^0(C,\cO_C)\cong H^0(C,\omega_C\otimes L_{\alpha}^{-1})$, then the map $H^0(C,E_{\alpha})\lra  H^0(C,\omega_C\otimes 
L_{\alpha}^{-1})$ is surjective and we get:
\begin{equation*}
 g-k'\le h^0(C,L_{\alpha})=h^0(C,\omega_C\otimes L_{\alpha}^{-1})+2g-2-e+1-g=g-e,
\end{equation*}
the claim follows. 
\vskip 3mm
Coming back to diagram (\ref{diagram}) we obtain 
\begin{equation*}
H^0(C,L_{\alpha})=H^0(C,\omega_C(-p_1-\dots -p_e)) \subset Ker (\cdot \cup \alpha)=\bigcap _{i=1,\dots,k'}H_i.
\end{equation*}
By dualizing, we obtain the inclusion of linear spans
\begin{equation*}
\langle x_1,\dots,x_{k'}\rangle \subset \langle p_1,\dots ,p_e \rangle .
\end{equation*}

\vskip 3mm
We denote by $\tC _0$ a non-degenerate irreducible curve, $\tC _0\ne C$, contained in all the quadrics parametrized by $K$. Let $\tC$ be the normalization of $\tC_0$ and let $\gamma $ be the normalization map. By choosing generically the $k+1$ points $x_i \in 
\tC$, we can assume that $k'$ and $e$ are constant, so the correspondence:
\begin{equation*}
 \Gamma =\{ (x_1+\dots +x_{k'},p_1+\dots+p_e) \,|\, \langle \gamma (x_1),\dots ,\gamma (x_{k'}) \rangle \subset \langle p_1,\dots,p_e \rangle \} \subset \tC^{\, (k')}\times C^{(e)}
\end{equation*}
 dominates $\tC^{\, (k')}$. Moreover, since $\tC$ is non-degenerate, the fibers of $\pi_2:\Gamma \lra C^{(e)}$ must be finite. Since $e\le k' \le \dim \Gamma =\dim 
 \pi_2(\Gamma)\le e$ we obtain that $e=k'$.
On the other hand the natural rational maps
\begin{equation*}
\begin{aligned}
& C^{(e)}\dasharrow Sec^e(C) \subset Grass(e-1,\mathbb P^{g-1})  \\
&  \tC^{(e)}\dasharrow Sec^e(\tC_0) \subset Grass(e-1,\mathbb P^{g-1})
\end{aligned}
\end{equation*}
are generically injective by the uniform position theorem (remember $e=k'\le k+1$ and $2k+5 < g$), hence the correspondence $\Gamma $ is of bidegree $(1,1)$ and therefore $\tC ^{\,(e)}$ and $C^{(e)}$ are birational.
In particular $g(\tC)=g(C)=g$ (the induced map on Jacobians $J\tC \lra JC$ has to be dominant since the image generates and the same in the opposite direction). By Ran's 
Theorem on symmetric products, see \cite{Ran}, we get $C\cong \tC$. Observe that both curves have to be canonical, hence $\tC_0$ is smooth and there is a linear projective transformation $\varphi :\tC\mapsto C$. 
Coming back to our argument and choosing $k$ generic points  $x_i \in \tC $ we have that $\langle x_1,\dots ,x_k \rangle = \langle \varphi (x_1),\dots ,
\varphi(x_k) \rangle$, hence $\varphi $ leaves $Sec^k (C)$ invariant and must be the identity, so $\tC=C$  which is a contradiction.

\vskip 3mm
\section{Divisor case, intersection with the boundary}

Now we start the proof of the codimension $1$ case of the main Theorem assuming $g\ge 5$. We put now $\cD$ instead of $\cZ$. The initial step of our degeneration procedure  
is to show that the intersection of $\cD$ with the boundary contains appropriate stable curves. These curves  have to contain enough information to deduce from them the 
main result for the general smooth curve. The goal of this section is to define a family of convenient reducible curves and to prove that they appear in the closure of 
$\cD$.

We start by recalling the following well-known facts on the rational Picard group of the compactified moduli space $ {\overline{\m{M}}_{g}}$ of stable curves (see for instance \cite{ACG}):
\[
Pic_{{\Q}}\,\m{M}_g=\lambda \Q,
\]
where $\lambda $ is the Hodge class. Moreover:
\[
 {\overline{\m{M}_g}} \setminus \m{M}_{g}= \bigcup_{i=0}^{[\frac{g-1}{2}]}\Delta_{i}
\]
and
\[
Pic_{\, \Q}\pa{{\overline{\m{M}_g}}}=\langle \lambda, \delta_{0},\delta_{1} , \ldots, \delta_{{[\frac{g-1}{2}]}}\rangle \, \Q
\]
\vskip 3mm
where $\delta_{i},\, i>0$ is the class of the divisor $\Delta_{i} $ whose general point represents a nodal curve $C_1 \cup C_2$, $C_1,\, C_2$ being integral, smooth curves 
of genus $i$ and $g-i$ intersecting in one point. And $\delta_0$ is the class of $\Delta_0$ whose general point represents an irreducible curve with exactly one node.

We denote by $d$ the class of $\overline{\mathcal D}$ in the rational Picard group. Then we can write:
\begin{equation}\label{expansion_d}
 d=a\lambda +\Sigma_{i\ge 0}a_i \delta _i.
\end{equation}

\begin{rem}\label{a_not_zero}
We note that $a$ must be different from zero. Otherwise the class of $\cD$ in the rational Picard group of $\cM_g$ would be zero.
Since the Satake compactification $\overline{\mathcal M}_g^s$ of $\cM_g$ is a projective variety and  the boundary $\overline{\mathcal M}_g^s - \cM_g$ has codimension $2$, given a smooth point $p$ of $\cD$,
there exists a complete curve $C$ in $\cM_g$ going through $p$ and cutting $\cD$ transversally. Hence $C\cdot \cD \ne 0$ which is a contradiction.
\end{rem}

Now we consider a complete integral curve $B$ in $\m{M}_{g-2}$ (it exists because $g-2\ge 3$) and  we fix two elliptic curves $E_ 1$, $E_2$ with arbitrary $j$-invariants $j_1,j_2\, \in \mathcal M_1$. 
Denote by $\Gamma_b$ the smooth curve of genus $g-2$ corresponding to $b$. We consider the set of the stable curves obtained by glueing to $\Gamma_b $ the two elliptic 
curves in two distinct points $p_1$ and $p_2$ of $\Gamma_b$. This does not depend on the choice of the points on the elliptic curves. This family is parametrized by the 
symmetric product $\Gamma_b^{(2)}\setminus \Delta_{\Gamma_b}$ minus the diagonal.
\vskip 3mm
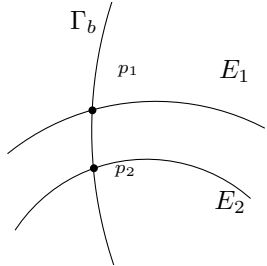
\begin{tikzpicture}[line cap=round,line join=round,>=triangle 45,x=1.0cm,y=1.0cm]
\clip(-3.00,2.16) rectangle (4.12,5.64);
\draw [shift={(7.25,3.94)}] plot[domain=2.83:3.55,variable=\t]({1*5.6*cos(\t r)+0*5.6*sin(\t r)},{0*5.6*cos(\t r)+1*5.6*sin(\t r)});
\draw [shift={(2.5,1.2)}] plot[domain=1.09:2.25,variable=\t]({1*3.13*cos(\t r)+0*3.13*sin(\t r)},{0*3.13*cos(\t r)+1*3.13*sin(\t r)});
\draw [shift={(2.38,1.47)}] plot[domain=0.85:2.56,variable=\t]({1*2.09*cos(\t r)+0*2.09*sin(\t r)},{0*2.09*cos(\t r)+1*2.09*sin(\t r)});
\draw (1.24,5.66) node[anchor=north west] {$ \Gamma_b $};
\draw (3.22,5) node[anchor=north west] {$E_1$};
\draw (3.16,3.26) node[anchor=north west] {$E_2$};
\begin{scriptsize}
\fill [color=black] (1.66,4.21) circle (1.5pt);
\draw[color=black] (2.14,4.74) node {$p_1$};
\fill [color=black] (1.68,3.44) circle (1.5pt);
\draw[color=black] (2.1,3.4) node {$p_2$};
\end{scriptsize}
\end{tikzpicture}
\vskip 3mm

So we have a well-defined map
\[
\begin{aligned}
 \Gamma_b^{(2)}\setminus \Delta_{\Gamma_b} &\lra \Delta _1 \subset \overline {\mathcal M_g} \\
p_1+p_2 &\mapsto E_1\cup_{p_1} \Gamma \cup_{p_2} E_2,
\end{aligned}
\]
  which extends to the whole symmetric product by sending $2p$ to the following curve: glue the infinity point of a $\mathbb P^1$ with the point $p$ and then glue $E_1$, 
  $E_2$ to other two points in the line. 
  \vskip 3mm
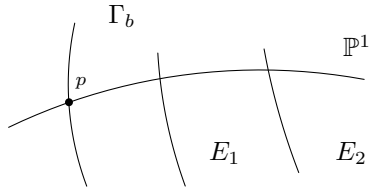
  \begin{tikzpicture}[line cap=round,line join=round,>=triangle 45,x=1.0cm,y=1.0cm]
\clip(-3.00,1.38) rectangle (6.72,4.84);
\draw [shift={(5.82,3.46)}] plot[domain=2.93:3.5,variable=\t]({1*3.87*cos(\t r)+0*3.87*sin(\t r)},{0*3.87*cos(\t r)+1*3.87*sin(\t r)});
\draw [shift={(4.5,-4.05)}] plot[domain=1.39:2.02,variable=\t]({1*7.7*cos(\t r)+0*7.7*sin(\t r)},{0*7.7*cos(\t r)+1*7.7*sin(\t r)});
\draw (5.48,4.24) node[anchor=north west] {$\mathbb P^1$};
\draw [shift={(9.39,4.24)}] plot[domain=3.2:3.49,variable=\t]({1*6.26*cos(\t r)+0*6.26*sin(\t r)},{0*6.26*cos(\t r)+1*6.26*sin(\t r)});
\draw [shift={(13.43,5.53)}] plot[domain=3.32:3.51,variable=\t]({1*9.02*cos(\t r)+0*9.02*sin(\t r)},{0*9.02*cos(\t r)+1*9.02*sin(\t r)});
\draw (2.36,4.64) node[anchor=north west] {$\Gamma_b$};
\draw (3.7,2.82) node[anchor=north west] {$E_1$};
\draw (5.38,2.82) node[anchor=north west] {$E_2$};
\begin{scriptsize}
\fill [color=black] (1.96,3.22) circle (1.5pt);
\draw[color=black] (2.12,3.48) node {$p$};
\end{scriptsize}
\end{tikzpicture}
  \vskip 3mm
  We note that these curves also belong to $\Delta_2$.

  Finally, by moving $b$ in the curve $B$ we obtain a complete threefold $T\subset \Delta_1$. In other words, this threefold can be seen as the image in $\overline {\mathcal 
  M_g}$ of the relative  symmetric product over $B$:
\[
 T=\bigcup_{b\in B} \Gamma_b^{(2)}.
\]
 
Our aim is to study the restriction of the divisor $\overline {\mathcal D}$ to $T$. To do this we make a computation in $Pic_{\Q}(T)$. Denote by $S$ the surface in $T$ 
obtained as the union of all the diagonals:
\[
 S=\bigcup_{b\in B} \Delta_{\Gamma_b}.
\]
 
We will need the following vanishing results.
\begin{lem}\label{lemma: delta_1_zero_on_S}
 The restriction of the class $\delta_1$ to $S$ is zero: $\delta_{1}|_S=0$.
\end{lem}

\begin{proof}
We fix a smooth curve $C$ of genus $2$ with a marked point $x$. We glue $C$ with $\Gamma_b$ identifying $x$ with $p\in \Gamma_b$. 
 Then, by moving $p$ in $\Gamma_b$ and $b$ in $B$, we construct an algebraic surface $S_C$ such that $ \Delta _1 \cap S_C=\emptyset$. Therefore $\delta_1 \cdot S_C=0$.
Now we degenerate $C$ to a genus $2$ curve with a marked point consisting in the two elliptic curves $E_1, E_2$ glued to a $\mathbb P^1$ in $0$ and $1$ respectively and infinity being the marked point. Therefore by adding the curve $\Gamma _b$ identifying $\infty$ with $p$ we get our surface $S$ as a limit of a family of algebraic surfaces $S_C$ as above. We get that $\delta_1\cdot S=0$.
\end{proof}

\begin{lem}
\label{lemma: lambda}
For each $b\in B$, $\lambda|_{\Delta_{\Gamma_b}}=0$.
\end{lem}
\begin{proof}
 The Hodge structure is constant along the diagonal.
\end{proof}

We also will use the following basic observation:
\begin{lem}
\label{lemma:nonzero}
Let $N$ be a complete curve in ${\bar{\m{M}}}_g$. Then $\rho|_{N}\neq 0$ for at least one class $\rho \in \{\lambda, \delta_0,\dots, \delta_{{[\frac{g-1}{2}]}}\}$.
\end{lem}

The main result of this section is the following
\begin{prop}\label{prop:limits}
 The restriction $d|_T$ is not a multiple of the class of $S$ in $Pic_{\Q}(T)$, i.e. $d|_T \ne m S$ for all $m\in \mathbb Q$. In particular $\overline {\mathcal D} \cap 
 T\ne \emptyset$ and this intersection contains elements outside $S$. 
\end{prop}
\begin{proof}
We use the notation introduced in (\ref{expansion_d}).
 By contradiction, assume that $d|_T = m S$. Notice that $\Delta_i $ does not intersect $T$ for $i=0$ and $i\ge 3$ and that $\Delta_{2}\cap T=S$, so we get:
\[
 d|_T=m S=a\lambda |_T + a_1 \delta _{1}|_T+a_2 k S,
\]
for some $k$. Therefore
\[
(m-a_2 k) S= a \lambda |_T +a_1 \delta _{1}|_T.
\]
Restricting to one diagonal $\Delta_{\Gamma_b}$ and using Lemmas \ref{lemma: delta_1_zero_on_S} and \ref{lemma: lambda} we deduce that $m-a_2 k=0$. Restricting now to 
$S$ we get $a \lambda |T=0$. Since $\lambda $ is not trivial on $T$ we obtain that $a=0$ which contradicts  Remark (\ref{a_not_zero}).
\end{proof}

\begin{rem}\label{set_of_limits}
Observe that the isomorphism classes of $E_1$ and $E_2$ are arbitrary, hence they could represent the $\infty$ class. Then the limit curves we were looking for are:
\[
 \mathcal L=\{E_{1}\cup_{p_1} \Gamma \, \cup_{p_2} E_{2} \in \overline {\mathcal D} \,|\, \Gamma \in \mathcal M_{g-2},\, p_1\ne p_2, E_1,E_2\in \overline {\mathcal M_1} \}
\subset \overline {\mathcal D}.
\]
 
\end{rem}

We recall that a generic point of $M_{g},\, g\ge 3$ is contained in a complete curve (see Remark (\ref{a_not_zero})). Then a consequence of the proposition (\ref{prop:limits}) is
\begin{cor}\label{cor:limits}
 There is a subvariety $\mathcal R \subset \mathcal M_{g-2}$ of codimension at most $1$  
such that for any $\Gamma \in \mathcal R$ and for any $E_1,E_2\in \overline 
 {\mathcal M_1} $, there are curves  $E_{1}\cup_{p_1} \Gamma \, \cup_{p_2} E_{2}$ in $\mathcal L$.
\end{cor}
Observe that all these elements belong to $\Delta_1$ and they belong to $\Delta_0$ if and only if at least one of the elliptic curves represents the infinity class.

\begin{rem}
With the same techniques one can prove that the divisor $\overline{\mathcal{D}}$ contains irreducible curves with only one node.
Moreover, by considering a surface in $\Delta_i$, $i>0$ given by fixed smooth curves of genus $g-i$ and $i$ and moving the intersection point, one also shows with the same procedure that the 
divisor contain ``generic'' elements of $\Delta_i$.  We do not use these facts in the rest of the paper.
\end{rem}

\vskip 3mm
\section{Limits of isogenies}
As in section 2 we assume the existence of isogenies $JC'\lra JC$ for generic elements $C\in \cD$. In order to glue all these maps together to provide a family $\chi: 
\m{J}' \lra \m{J}$ we need to pull-back to a suitable branched cover of our divisor. Since our calculations will be mainly of local nature we still denote this space of 
parameters as $\cD$. 
Our goal is to get as much information as possible from the specialization of the family of isogenies to the curves of $\mathcal L$.  

Let $L=E_{1}\cup_{p_1} \Gamma \, \cup_{p_2} E_{2} \in \mathcal L \subset \overline {\mathcal D}$ a fixed limit curve (see Remark (\ref{set_of_limits}) and Corollary 
(\ref{cor:limits})). We assume that one of the elliptic curves has a node. 

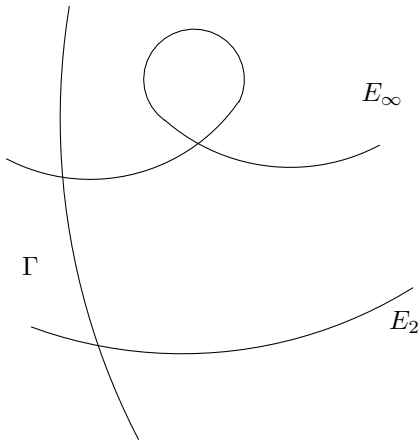
\begin{tikzpicture}[line cap=round,line join=round,>=triangle 45,x=1.0cm,y=1.0cm]
\clip(-3.15,-3.48) rectangle (7.75,3.78);
\draw [shift={(1.31,2.97)}] plot[domain=4.23:5.68,variable=\t]({1*2.4*cos(\t r)+0*2.4*sin(\t r)},{0*2.4*cos(\t r)+1*2.4*sin(\t r)});
\draw [shift={(2.69,1.9)}] plot[domain=-0.46:4.09,variable=\t]({1*0.67*cos(\t r)+0*0.67*sin(\t r)},{0*0.67*cos(\t r)+1*0.67*sin(\t r)});
\draw [shift={(3.97,3.26)}] plot[domain=3.99:5.2,variable=\t]({1*2.53*cos(\t r)+0*2.53*sin(\t r)},{0*2.53*cos(\t r)+1*2.53*sin(\t r)});
\draw [shift={(10.19,1.38)}] plot[domain=2.98:3.62,variable=\t]({1*9.28*cos(\t r)+0*9.28*sin(\t r)},{0*9.28*cos(\t r)+1*9.28*sin(\t r)});
\draw [shift={(2.53,4.05)}] plot[domain=4.36:5.27,variable=\t]({1*5.8*cos(\t r)+0*5.8*sin(\t r)},{0*5.8*cos(\t r)+1*5.8*sin(\t r)});
\draw (4.79,1.98) node[anchor=north west] {$ E_{\infty} $};
\draw (0.28,-0.34) node[anchor=north west] {$ \Gamma $};
\draw (5.16,-1.06) node[anchor=north west] {$ E_2 $};
\end{tikzpicture}

Observe that $L\in \Delta_0 \cap \Delta_1$. Let $\overline {\mathcal D_0}$ be a component of $\overline{\mathcal D} \cap \Delta_0$ containing $L$.  Since we can assume that 
the generic element of this component does not belong to $\Delta_i$, $i\ge 2$, we obtain the following four cases:
\begin{enumerate}
 \item[a)] The generic element of $\overline {\mathcal D_0}$ is an irreducible curve with only one node.
 \item[b)] The generic element of $\overline {\mathcal D_0}$ consists of an irreducible curve with only one node with an elliptic curve attached in a smooth point.
 \item[c)] The generic element of $\overline {\mathcal D_0}$ consists of a smooth irreducible curve with an elliptic nodal curve attached in a smooth point.
 \item[d)] The generic element of $\overline {\mathcal D_0}$ consists of two irreducible smooth curves glued in two different points.
 \item[e)] The generic element of $\overline {\mathcal D_0}$ is an irreducible curve with two nodes.
\end{enumerate}

For topological reasons the cases d) and e) can not occur. Indeed, the number of nodal points in a stable curve such that the curve remains connected when we remove the 
point can not decrease under specialization. Since there are two such a points in a generic point of type d) and e) and only one in our limit curve $L$ we can ignore these cases.

In this section we will specialize the isogeny $\chi$ to a generic curve of the component $\overline {\mathcal D_0}$ that contains our limit curve $L$.

\begin{rem}
We will use several times that the limit of a family of isogenies of Jacobians is also an isogeny, that is, an \'etale surjective map between the generalized Jacobians. We briefly give an argument to justify this. Consider two semistable flat families of curves parametrized by a disk  $\pi ':\mathcal C'\lra \mathbb D$ and  $\pi: \mathcal C\lra \mathbb D$ such that the curves are smooth away from the central fibres $\mathcal C'_{0}$, $\mathcal C_{0}$. Given a family of isogenies $\varphi: \mathcal {JC}' \lra \mathcal {JC}$ over the punctured disk, the existence of a map $\varphi_0:\mathcal {JC}'_0 \lra  \mathcal {JC}_0$ of semiabelian varieties is given in \cite{BarPir}, p. 267. We want to prove that this map is also an isogeny. Recall that the family of Jacobians is given by
\[
 \mathcal {JC'}=R^1\pi_* \mathcal O / R^1\pi_* \mathbb Z
\]
and analogously for $\mathcal {JC}$. The fact that $\varphi_t$ is an isogeny for $t\ne 0$ implies that there is a sequence over the punctured disk
\[
0\lra R^1 \pi'_{\ast}\mathbb Z \lra R^1 \pi_{\ast} \mathbb Z \lra \tau \lra 0,
\]
where $\tau $ is a torsion local system. Hence tensoring with $\mathbb Q$ we obtain an isomorphism 
 \[ 
 R^1 \pi_*' \mathbb Q \cong R^1\pi_* \mathbb Q. 
\]
Since the rational cohomology of the central fibre can be identified with the invariant subspace by the action of the monodromy we have that:
\[
 H^1(\mathcal C_0',\mathbb Q)\cong H^1 (\mathcal C'_{|\mathbb D^*},\mathbb Q)^\text{inv}\cong H^0(\mathcal C'_{|\mathbb D^*}, R^1 \pi_*' \mathbb Q),
\]
and therefore we conclude that $H^1(\mathcal C_0',\mathbb Q)\cong H^1(\mathcal C_0,\mathbb Q)$. So the map $\varphi_0$ is an isogeny between the generalized Jacobians.
\end{rem}

When we go to the limit the information we obtain is different according to the three cases explained above.

{\bf Case a).}
We consider the normalization map $p:\Delta_0 \dashrightarrow \m{M}_{g-1}$ restricted to $\overline{\cD_0}$.  
There are two possibilities according to the dimension of the generic fiber of $p_0:=p|_{\overline {\cD_0}}$.

\vskip 3mm
{\bf Case a.1):} Assume that $p_0$ is dominant, therefore the generic fibre has dimension $dim \overline{\cD_0}-\dim \m{M}_{g-1}=3g-5-3g+6=1$. For a generic element 
$t_0\in \overline{\cD_0}$ the limit map 
$\chi_{t_0}:JC'_{t_0}\lra JC_{t_0}$ gives a diagram of extensions:
\[
 \xymatrix@C=1.pc@R=1.8pc{
0 \ar[r]     &\mathbb C^{*\, r}\ar[r] \ar[d]^{\gamma}   & JC'_{t_0} \ar[r] \ar[d]^{\chi_{t_0}}  & J\tilde {C'}_{t_0} \ar[r] \ar[d]^{\tilde \chi _{t_0} }  & 0 \\
0 \ar[r]     &\mathbb C^*     \ar[r]           & JC_{t_0}  \ar[r]         & J\tilde {C}_{t_0} \ar[r]           & 0
}
\]
where $\tilde C'_{t_0}$ and $\tilde C_{t_0}$ stand for the normalizations of $C'_{t_0}$ and $C_{t_0}$ respetively. Since $\chi_{t_0}$ has finite kernel,
$r$ must be one 
and $\gamma (z)=z^m$ for some non-zero integer $m$. Since $g-1\ge 4$ we can apply the main result in \cite{BarPir} and we get that $\tilde C'_{t_0}=\tilde C_{t_0}$ and the 
isogeny $\tilde \chi _{t_0}$
is $n$ times the identity. Assume that $C_{t_0}$ (resp. $C'_{t_0}$) is obtained from $\tilde C_{t_0}$ by pinching two distinct points $p,q$ (resp. $p',q'$).  
As in \cite{BarPir}, section 2, to compare the extension classes of each horizontal short exact sequence, we decompose the last diagram into
\[
 \xymatrix@C=1.pc@R=1.8pc{
0 \ar[r]     &\mathbb C^* \ar[r] \ar[d]^{\gamma}   & JC'_{t_0} \ar[r] \ar[d]^{\chi_{t_0}}  & J\tilde {C}_{t_0} \ar[r] \ar@{=}[d]   & 0 \\
0 \ar[r]     &\mathbb C^*  \ar[r]   \ar@{=}[d]  &E \ar[r] \ar[d] &  J\tilde {C}_{t_0}\ar[d]^{\tilde \chi _{t_0} } \ar[r] &0 \\
0 \ar[r]     &\mathbb C^*     \ar[r]           & JC_{t_0}  \ar[r]         & J\tilde {C}_{t_0} \ar[r]           & 0.
}
\]

We identify (up to sign) the extension class 
\[
 [JC'_{t_0}]\in Ext(J\tilde {C}_{t_0},\mathbb C^*)\cong Pic^0(J\tilde {C}_{t_0})\cong J\tilde {C}_{t_0}
 \]
 with $p'-q'$
and analogously $[JC_{t_0}]$ with $p-q$. Then the equality $[E]=\gamma _*([JC'_{t_0}])=\tilde \chi _{t_0} ^*([JC_{t_0}])$ provides
the following  relation in $J\tilde C_{t_0}$:
\[
 n(p-q)=\pm \, m(p'-q').
\]
 Hence we can assume that $np+mq'=mp'+nq$  in $Pic(\tilde C_{t_0})$. We assume also that the points are different. Since the dimension of the generic fibre of $p_{0}$ is 
 $1$ we have a one dimensional family of maps $\tilde C_{t_0}\lra \mathbb P^1$ of degree $n+m$ with two  fibers as above. The Riemann-Hurwitz Theorem implies that
\[
 2g(\tilde C_{t_0})-2= 2g-4=(n+m)(2g(\mathbb P^1)-2)+ 2(n-1)+2(m-1)+r=-4+r,
\]
 so the number $r$ of the ramification points out of the special fibers $np+mq'$ and $mp'+nq$ is $r=2g$. Then the Hurwitz scheme of maps of degree $n+m$ into $\mathbb P^1$ 
 with $r+2=2g+2$ discriminant points must cover $\m{M}_{g-1}$ with generic fibers of dimension $1$. Comparing dimensions:
\[
2g+2-\dim \text{Aut}(\mathbb P^1)-\dim \text {generic fiber}=2g+2-4=2g-2 \ge \dim \m{M}_{g-1}=3g-6, 
\]
 which contradicts the hypothesis $g\ge 5$. Hence we get that the extension is the same and $n=m$. 

\vskip 3mm
{\bf Case a.2):} Assume that the generic fiber of $p_{0}$ has dimension $2$. 
As before we get a diagram:
\[
 \xymatrix@C=1.pc@R=1.8pc{
0 \ar[r]     &\mathbb C^*     \ar[r] \ar[d]^{m}  & JC'_{t_0} \ar[r] \ar[d]^{\chi_{t_0}}  & J\tilde {C'}_{t_0} \ar[r] \ar[d]^{\tilde \chi _{t_0} }  & 0 \\
0 \ar[r]     &\mathbb C^*     \ar[r]                  & JC_{t_0}  \ar[r]         & J\tilde {C}_{t_0} \ar[r]           & 0
}
\]
but now we do not have the genericity of $\tilde C_{t_0}$ so we can not directly apply the main result in \cite{BarPir}. The relation between extension classes is in this case
\[
 m(p'-q')=\tilde \chi _{t_0}^* (p-q),
\]
in $J\tilde {C'}_{t_0}$. In other words, the isogeny $\tilde \chi _{t_0}$ induces a map between the surfaces $\tilde C_{t_0}-\tilde C_{t_0}$
and $m(\tilde C'_{t_0}-\tilde C'_{t_0})$. 
By using the arguments of the section $3$ in \cite{BarPir} one easily checks that, as before, the curves are the same and the map is a multiplication by an integer. So we have 
proved the following result:

\begin{prop}\label{nodal_limit}
 Let $\chi:\cJ' \lra \cJ$ be a family of isogenies parametrized by $\cD$ and let $\chi_{t_0}:JC'_{t_0}\lra JC_{t_0}$ be a specialization 
to a generic point $t_0$ of a component of the boundary $\overline{\cD}\cap \Delta_0$, where the curve $C_{t_0}$ is an irreducible curve with only one node. Then 
$C'_{t_0}\cong C_{t_0}$ and $\chi_{t_0}$ is the multiplication by a non-zero integer. 
\end{prop}

{\bf Case b):}
We assume now that the limit curve $L$ belongs to an irreducible component $\overline{\cD_0}$ of $\overline {\cD}\cap \Delta _0$ whose generic element consists in a smooth 
curve $\Gamma$ of genus $g-1$ and a nodal elliptic curve $E_{\infty}$ (i.e. a $\mathbb P^1$ with the points $0$ and $1$ identified) glued to $\Gamma $ in a point $p\in 
\Gamma$. We denote as above by $C_{t_0}=\Gamma \cup _p \,E_{\infty}$ a generic curve in $\overline{\cD_0}$. Observe that the natural map $\overline{\cD_0}\dasharrow 
\mathcal M_{g-1}$ must be dominant by a count of dimensions (the fibre has dimension at most $1$), therefore we can assume that $\Gamma $ is generic in $\mathcal M_{g-1}$.

We consider the specialization
of the family of isogenies to our curve $\chi_{t_0}:JC'_{t_0}\lra JC_{t_0}=J\Gamma \times \mathbb C^{*}$ which fits in a diagram of 
extensions:
\[
 \xymatrix@C=1.pc@R=1.8pc{
0 \ar[r]     &\mathbb C^{*\, r}\ar[r] \ar[d]^{\gamma}   & JC'_{t_0} \ar[r] \ar[d]^{\chi_{t_0}}  & J\tilde {C'}_{t_0} \ar[r] \ar[d]^{\tilde \chi _{t_0} }  & 0 \\
0 \ar[r]     &\mathbb C^*     \ar[r]           & J\Gamma \times \mathbb C^{*}  \ar[r]         & J\Gamma \ar[r]           & 0
}
\]
where  $\tilde C'_{t_0}$ stands for the normalization of $C'_{t_0}$. Since $\chi_{t_0}$ has finite kernel,
$r$ must be one and $\gamma (z)=z^m$ for some non-zero integer 
$m$. Since $g-1\ge 4$ we can apply the main result in \cite{BarPir} and we get that $\tilde C'_{t_0}=\Gamma$ and the isogeny $\tilde \chi _{t_0}$
is $n$ times the identity. So the diagram above becomes:
\[
 \xymatrix@C=1.pc@R=1.8pc{
0 \ar[r]     &\mathbb C^{*}\ar[r] \ar[d]^{m}   & JC'_{t_0} \ar[r] \ar[d]^{\chi_{t_0}}  & J\Gamma \ar[r] \ar[d]^{n }  & 0 \\
0 \ar[r]     &\mathbb C^*     \ar[r]           & J\Gamma \times \mathbb C^{*}  \ar[r]         & J\Gamma \ar[r]           & 0
}
\]
Since the extension class of the first row corresponds to a generalized Jacobian, there exist points $q_1,q_2 \in \Gamma$ such that this class corresponds (up to sign) 
to $q_1-q_2 \in J\Gamma$. Therefore, since the class of the second row is zero we get that $m(q_1-q_2)=0$.

Moving the point $p$ in $\Gamma $ we have a positive dimensional family of pairs of points $(q_1,q_2)\in C\times C$ with this property. This family has to be the diagonal since the map $\Gamma \times \Gamma \lra J\Gamma $ has degree $1$ and finite fibres out of the diagonal. Hence  we obtain $q_1=q_2$.

Hence the extension given by the first row is also trivial: $JC'_{t_0} \cong J\Gamma \times \mathbb 
C^{*}$ and $\chi_{t_0}=\begin{pmatrix}n&0 \\ 0 &m \end{pmatrix}$. 
We get the following result:

\begin{prop}\label{limit_b}
 Let $\chi:\cJ' \lra \cJ$ be a family of isogenies parametrized by $\cD$ and let $\chi_{t_0}:JC'_{t_0}\lra JC_{t_0}$ be a specialization 
to a generic point $t_0$ of a component of the boundary $\overline{\cD}\cap \Delta_0$, where the curve $C_{t_0}$ is a reducible curve consisting in a smooth curve of 
genus $g-1$ with a nodal elliptic curve $E_{\infty}$ attached in one point. Then $C'_{t_0}\cong C_{t_0}$ and the isogeny in the compact part is the 
multiplication by a non zero integer. 
\end{prop}

{\bf Case c):}
Finally we assume that the limit curve $L$ belongs to an irreducible component $\overline{\cD_0}$ of $\overline {\cD}\cap \Delta _0$ whose generic element consists in a 
nodal curve $\Gamma_0$ of genus $g-1$ and an elliptic curve $E$ glued to $\Gamma _0$ in a smooth point $p\in \Gamma_0$. We denote as above by $C_{t_0}=\Gamma _0\cup _p 
\,E$ a generic curve in $\overline{\cD_0}$. 
Observe that the natural map $\overline{\cD_0}\dasharrow \Delta_0(\mathcal M_{g-1})$  must be dominant (here $\Delta_0(\mathcal M_{g-1})$ denotes the $\Delta_0$ divisor in the moduli space $\mathcal M_{g-1}$). Indeed the generic fibre of this map has dimension at most $2$ 
parametrized by the smooth point $p$ in $\Gamma_0$ and the moduli of the elliptic curves. Since $\dim \overline{\cD_0}=3g-5$ and $\dim  \Delta_0(\mathcal M_{g-1})=3g-7$, 
the dominance follows. Therefore we can assume that $\Gamma _0$ is generic in $\Delta_0(\mathcal M_{g-1})$, and also its normalization $\tilde \Gamma _0$ is generic in 
$\mathcal M_{g-2}$. Moreover all the curves $\Gamma _0\cup _p \,E$ are contained in $\overline{\cD_0}$ for a generic $\Gamma _0$.

As in the other cases we consider the specialization 
of the family of isogenies to our curve $\chi_{t_0}:JC'_{t_0}\lra JC_{t_0}=J\Gamma _0 \times E$ which fits in a 
diagram of extensions:
 \[
  \xymatrix@C=1.pc@R=1.8pc{
 0 \ar[r]     &\mathbb C^{*\, r}\ar[r] \ar[d]^{\gamma}   & JC'_{t_0} \ar[r] \ar[d]^{\chi_{t_0}}  & J\tilde {C}'_{t_0} \ar[r] \ar[d]^{\tilde \chi _{t_0} }  & 0 \\
 0 \ar[r]     &\mathbb C^*     \ar[r]           & J\Gamma _0\times E \ar[r]         & J\tilde{\Gamma}_0\times E \ar[r]           & 0
}
 \]
 where  $\tilde C'_{t_0}$ stands for the normalization of $C'_{t_0}$. Since $\chi_{t_0}$ has finite kernel, $r$ must be one and $\gamma (z)=z^m$ for some non-zero integer 
 $m$. 
 
 We claim that $J\tilde {C'}_{t_0}$ must be a product of Jacobians (in other words, the smooth curve $\tilde C'_{t_0}$ is reducible). We prove this by contradiction, assume 
 that $\tilde C'_{t_0}$ is irreducible and compare the extension classes. This gives (up to sign) a relation $m (p'-q')=\tilde \chi_{t_0}(p-q)$ in $J\tilde 
 {C'}_{t_0}$. Moving the points $p$ and $q$ in the fixed  curve $\tilde \Gamma _0$ we get that the image of $\tilde \Gamma _0 - \tilde \Gamma _0$ by the isogeny is the 
 surface $\tilde {C'}_{t_0}-\tilde {C'}_{t_0}$ which is impossible since it is contained in the proper abelian subvariety $ \tilde \chi_{t_0}^*(J\tilde \Gamma _0)$.
 
 By the genericity of $\tilde \Gamma _0$ we can assume that $J\tilde \Gamma _0$ is simple and then  $J\tilde C'_{t_0}\cong  J\tilde C''_{t_0}\times E'$, where $\tilde 
 C''_{t_0}$ is an irreducible curve of genus $g-2$ and $E'$ stands for a smooth elliptic curve. So $C'_{t_0}$ is a nodal curve $C''_{t_0}$ with the elliptic curve attached 
 in a smooth point. The extension class is the difference of two points $p'', q''\in \tilde C''_{t_0}$.
 The diagram above becomes:
 \[
  \xymatrix@C=1.pc@R=1.8pc{
 0 \ar[r]     &\mathbb C^{*}\ar[r] \ar[d]^{m}   & JC''_{t_0}\times E' \ar[r] \ar[d]^{\chi_{t_0}}  & J\tilde C''_{t_0}\times E' \ar[r] \ar[d]^{\tilde \chi_{t_0}=\tilde 
 \chi''_{t_0}\times \varphi } & 0 \\
 0 \ar[r]     &\mathbb C^*     \ar[r]           & J \Gamma _0\times E  \ar[r]         & J\tilde \Gamma _0\times E\ar[r]           & 0,
}
 \]
 where $\varphi: E' \lra E$ is a non-constant map of elliptic curves.
The relation between extension classes is in this case (up to sign):
\[
 m(p''-q'')=(\tilde \chi'' _{t_0})^* (p-q),
\]
in $J\tilde C''_{t_0}$. In other words, the isogeny $\tilde \chi ''_{t_0}$ induces a map between the surfaces $\tilde \Gamma _0-\tilde \Gamma _0$
and $m(\tilde C''_{t_0}-\tilde C''_{t_0})$. 
By using the arguments of Section $3$ in \cite{BarPir} one easily checks the following facts: the curves $C''_{t_0}$ and $\Gamma _0$ are isomorphic, the isogeny $\tilde  
\chi''_{t_0}$ is a non-zero multiple $n$ of the identity, $n=m$ and then $\chi_{t_0}=n\times \varphi $.

\begin{prop}\label{limit_c}
 Let $\chi:\cJ' \lra \cJ$ be a family of isogenies parametrized by $\cD$ and let $\chi_{t_0}:JC'_{t_0}\lra JC_{t_0}$ be a specialization 
to a generic point $t_0$ of a component of the boundary $\overline{\cD}\cap \Delta_0$, where the curve $C_{t_0}$ is a reducible curve consisting in a nodal curve 
$\Gamma _0$ of genus $g-1$ with an elliptic curve $E$ attached in one smooth point, then also  $C'_{t_0}$ is of the form $\Gamma _0\cup_p E'$, for some elliptic curve 
$E'$, and the isogeny induces the multiplication by $n$ on $J\Gamma _0$. 
\end{prop}

\vskip 3mm
\section{End of the proof} 
Let us go back to our family of isogenies $\chi:\mathcal J' \lra \mathcal J$ parametrized by (some covering of) the divisor  $\mathcal D$.

Consider a generic point $t\in \cD$ corresponding to smooth curves $C_t'$ and $C_t$. 
Observe that for all $t$ the isogeny is determined by the map at the level of homology groups 
\[
 \chi_{t,\mathbb Z}:H_1(C'_t, \mathbb Z) \lra H_1 (C_t,\mathbb Z) 
\]
which we still denote by $\chi_t$. We set $\Lambda_t\subset H_1(C_t,\mathbb Z)$ for the image of $\chi_t$. This is a sublattice of maximal rank $2g$. We first note 
that the proof of Proposition (4.2.1) in \cite{BarPir} applies verbatim  to obtain the following result:

\begin{prop}
 Assume that $\Lambda_t =n H_1(C_t,\mathbb Z)$ for some positive integer $n$. Then $C_t'\cong C_t$ and $\chi _t$ is multiplication by $n$.
\end{prop}

Therefore to finish the proof of the Theorem we have to show the equality  $\Lambda_t =n H_1(C_t,\mathbb Z)$. To do this, the main idea is to get information on 
$\Lambda_t$ from the homology groups of some convenient limits $C_0$. In the previous sections we have shown the existence of certain limits and we have seen in Propositions (\ref{nodal_limit}), (\ref{limit_b}) and 
(\ref{limit_c})  how is the image of the limit of $ \chi_{t}$ when $t$ goes to one of these degenerations.

To pass information from some $C_0$ to the smooth point we use the following principle: 
 We can assume that there exists a disc $\mathbb  D\subset \overline {\mathcal D}\subset \overline{\mathcal M_{g}}$ centered at the class of the curve  $C_0$  such that the curves $C_t, C_t'$ corresponding to $\mathbb D \setminus \{0\}$ are smooth. After performing a base change (that we skip to simplify the notation) we can assume
that there is a family of 
 isogenies $\chi_{\mathbb D}:\mathcal J_{\mathbb D}'\lra \mathcal J_{\mathbb D}$ that coincides with the original isogeny $\chi_t$ for a generic $t$.
  We denote $\chi_0: JC_0' \lra JC_0$ the limit isogeny. 
We call $\mathcal C'_{\mathbb D}$ and  $\mathcal C_{\mathbb D}$ the corresponding families of curves. Since the central fibres $C_0'$ and $C_0$ are retracts of $\mathcal C'_{\mathbb D}$ and  $\mathcal C_{\mathbb D}$ respectively, we have a 
diagram   as follows:
\[
\xymatrix@C=1.pc@R=1.8pc{
H_1(C'_t,\mathbb Z)\ar[d]^{\chi_t}  \ar[r] & H_1(\mathcal C'_{\mathbb D},\mathbb Z) \ar@{=}[r] & H_1(C_0',\mathbb Z) \ar[d]^{\chi_0} \\
H_1(C_t,\mathbb Z)                  \ar[r] & H_1(\mathcal C_{\mathbb D},\mathbb Z)  \ar@{=}[r]& H_1(C_0,\mathbb Z).  
}
\]
Let us consider first the simplest case: assume that we are in the case a) and that $C_0$ is a generic element of $\overline {\mathcal D}\cap \Delta_0$, so  it is an irreducible curve with 
only one node. Remember that $C_0$ can be degenerated in this component to a curve $L$ consisting of a curve of genus $g-2$ with two attached curves, one elliptic (called $E_2$) and the other nodal and rational ($E_{\infty }$).  By Proposition (\ref{nodal_limit}) $C_0'\cong C_0$ and $\chi_0=n\cdot \text{Id}$.
The kernel of the horizontal maps are generated by the vanishing cycles $a_1'$ and $a_1$ respectively, therefore $\chi_t(a_1')=la_1$ for  some $l$.
 On the other hand we 
can lift a basis in $H_1(C_0,\mathbb Z)$  and construct simplectic bases $a'_1,b'_1,\dots,a'_g,b'_g$ in $H_1(C'_t,\mathbb Z)$ and $a_1,b_1,\dots ,a_g,b_g$ in $H_1(C_t,
\mathbb Z)$ in such a way that:
\[
 \chi _s(a_i')=n a_i+s_i a_1, \quad \chi _t( b_1') = n b_1+t_1a_1,\quad \chi _t(b_i') = n b_i+t_ia_1
\]
for some integers $s_i$ and $t_i$ and $i\ge 2$.   We also can assume that $a_g,b_g$ correspond to cycles which become a basis of the homology of $E_2\subset L$.

By the genericity of the curve $C_t$ in a divisor of the moduli space the pull-back of the theta divisor is a multiple of the theta divisor in $JC'_t$. This translates into the existence of a non-zero integer $m$ such that the cup-product satisfies 
\[
 \chi _t(x)\cup \chi _t(y)=m x\cup y.
\]
Then we obtain
\[
\begin{aligned}
 &m=\chi _t(a_1')\cup \chi_t(b_1')= la_1 \cup (nb_1+t_1a_1)=ln \\
&  m=\chi_t(a_2')\cup \chi _t(b_2')=n^2,
\end{aligned}
\]
so $n=l$. With similar computations it is easy to prove that $s_i=t_i=0$ for $i\ge 2$, hence:
\begin{equation}\label{lattice_case_aa}
 \chi_t(a_1')=n a_1,\qquad \chi_t(a_i')=na_i,\qquad \chi_t(b_i')=nb_i \qquad \text { for }\, i\ge 2.
\end{equation}
To get the piece of information which is still unknown we need to consider a second limit. We have seen in Corollary (\ref{cor:limits}) that fixing $\Gamma$ we can 
move freely the two elliptic curves $E_1$, $E_2$, that is $E_1 \cup \Gamma \cup E_2 \in \overline{\cD}$ for all $E_1, E_2$. We select a second limit curve $\hat L =E_1 
\cup \Gamma \cup E_{\infty}$ in such a way that the corresponding vanishing cycle is now $a_g$. Again, to simplify, we assume that $\hat L$ belongs to the case a) of 
section 5. Then, with the same argument, we get that $\chi_t$ satisfies (for the same simplectic basis):
\[
  \qquad \chi_t(a_i')=\hat na_i,\qquad \chi_t(b_i')=\hat nb_i \qquad \text { for }\, i\le  g-1
\]
and $\chi_t(a_g')=\hat n a_g$. Therefore $\hat n=n$ and $\Lambda_t =n H_1(C_t,\mathbb Z)$ for all $s\ne 0$. Hence we have finished (under the assumption that $L$ 
and $\hat L$ are limit curves of nodal curves). The rest of the proof consists in the description of the small modifications that have to be done to take care of the rest 
of the cases. Observe that the information on the limit given in the case b) (see Proposition (\ref{limit_b})) is the same as that given in the case a), that is we have again the relations (\ref{lattice_case_aa}). So we  only have 
to take care of the situation when at least one of the limit curves $L$, $\hat L$ belong to the case c).
Assume for example that $L$ does.
Using the first limit as above we know that
\[
 \chi_t(a_1')=n a_1,\qquad \chi_t(a_i')=na_i,\qquad \chi_t(b_i')=nb_i \qquad \text { for }\, 2\le i \le g-1.
\]
The difference with the previous cases is that we have no control on $\chi_t(b_1')$.
Remember that $L$ is the limit of a curve $\Gamma_0 \cup  E$ where $\Gamma_0$ is a nodal curve intersecting the elliptic curve $E$ in a smooth point. Denoting $\tilde 
\Gamma_0$ the normalization of $\Gamma_0$ we note that all the nodal curves with this normalization belong to the divisor $\overline {\mathcal D_0}$. So we change 
freely the node in such a way that the vanishing cycle of the node becomes $a_2$ (instead of $a_1$). By using the two vanishing cycles we get that $\chi_t(b_1')=nb_1'$ so we recover again the relations (\ref{lattice_case_aa}). This finishes the proof of the theorem. \qed

\vskip 3mm
Our theorem can be interpreted as a type of Noether-Lefschetz problem in the following way: consider in $\m{M}_g \times \m{M}_g$ the set
\[
 \mathcal {NL}_g=\{(C',C) \,|\, \text{ rank }NS(C  \times C')\ge 3\}.
\]
A consequence of what we have proved is the following result:

\begin{cor}\label{remark_NL}
For $g\ge 5$ all the components of $\mathcal {NL}_g$ outside the diagonal have dimension less than or equal to $3g-5$.
\end{cor}
The first natural problem one could face in this context is to investigate the existence of dimension $10$ components in $\mathcal {NL}_5$.  
Similar problems on isogenies can be considered for other families of abelian varieties (see for exemple \cite{NarPir}).

\end{document}